\documentclass[a4paper,english,oneside]{article}
\usepackage[T1]{fontenc}
\usepackage[utf8]{inputenc} 

\usepackage{geometry}
\usepackage{amsmath}

\usepackage{amsfonts}
\usepackage{amssymb}
\usepackage{amsthm}

\usepackage{pgfplots}
\usepackage{tkz-euclide}

\usepackage[english]{babel}
\usepackage{color}
\usepackage{lmodern}
\usepackage{pstricks-add}
\usepackage{mathptmx}
\usepackage{hyperref}
\hypersetup{hidelinks}

\usepackage{mathscinet}
\usepackage{enumitem}
\usepackage[cal=boondoxo,bb=ams]{mathalfa}
\usepackage{microtype}

\usepackage{mathtools}

\theoremstyle{plain}
\newtheorem{theorem}{Theorem}[section]
\newtheorem{lemma}[theorem]{Lemma}
\newtheorem{corollary}[theorem]{Corollary}

\theoremstyle{definition}
\newtheorem{definition}[theorem]{Definition}

\theoremstyle{remark}
\newtheorem{remark}{Remark}

    \DeclareMathOperator\supp{supp}

    \newcommand{\mb}{\mathbb}
    \newcommand{\mm}{\mathrm}

\begin{document}

\title{Nonexistence of global solutions for generalized Tricomi equations with combined nonlinearity}



\author{Wenhui Chen$^\mathrm{a}$, Sandra Lucente$^\mathrm{b}$, Alessandro Palmieri$^\mathrm{c}$
}

\date{\small{$^\mathrm{a}$ School of Mathematical Sciences, Shanghai Jiao Tong University, 200240 Shanghai, China}\\
\small{$^\mathrm{b}$ Department of Physics, University of Bari, Via E. Orabona 4, 70125 Bari, Italy}\\
 \small{$^\mathrm{c}$ Department of Mathematics, University of Pisa,  Largo B. Pontecorvo 5, 56127 Pisa, Italy}
}

\maketitle

\begin{abstract}

In the present paper, we investigate the blow-up dynamics for local solutions to the semilinear generalized Tricomi equation with combined nonlinearity. As a result, we enlarge the blow-up region in comparison to the ones for the corresponding semilinear models with either power nonlinearity or nonlinearity of derivative type. Our approach is based on an iteration argument to establish lower bound estimates for the space average of local solutions. Finally, we obtain upper bound estimates for the lifespan of local solutions as byproduct of our iteration argument.

\end{abstract}

\begin{flushleft}
\end{flushleft}

\begin{flushleft}
\textbf{AMS Classification (2020)} Primary:  35B44, 35L71; Secondary: 35L05
\end{flushleft}

\section{Introduction}

Aim of the present work is to derive a blow-up result for the semilinear \emph{generalized Tricomi equation} with a combined nonlinearity
\begin{align}\label{Semi_Tri_Combined}
\begin{cases} 
\partial_t^2 u- t^{2\ell}\Delta u =|\partial_t u|^p+|u|^q, &  x\in \mathbb{R}^n, \ t>0,\\
u(0,x)=\varepsilon u_0(x), & x\in \mathbb{R}^n, \\ 
\partial_t u(0,x)=\varepsilon u_1(x), & x\in \mathbb{R}^n,
\end{cases}
\end{align} 
where $\ell$ is a positive parameter, $p,q>1$ and $\varepsilon$ is a positive constant describing the size of Cauchy data. More specifically, our goal is to enlarge the blow-up region in the $(p,q)$ - plane in comparison to the ones for the corresponding models with \emph{power nonlinearity} $|u|^q$ and \emph{nonlinearity of derivative type} $|\partial_t u|^p$, respectively.

Over the last years, several papers have been devoted to the study of semilinear Cauchy problem associated to the \emph{generalized Tricomi operator} $\partial_t^2-t^{2\ell} \Delta$. In \cite{HWY17} the authors conjectured the critical exponent for the semilinear generalized Tricomi equation with power nonlinearity in space dimension $n\geqslant 2$
\begin{align}\label{Semi_Tri_Pow_Non}
\begin{cases} 
\partial_t^2 u- t^{2\ell}\Delta u =|u|^p, &  x\in \mathbb{R}^n, \ t>0,\\
u(0,x)=\varepsilon u_0(x), & x\in \mathbb{R}^n, \\ 
\partial_t u(0,x)=\varepsilon u_1(x), & x\in \mathbb{R}^n,
\end{cases}
\end{align} where $\ell>0$ and $p>1$. More precisely, this exponent, denoted by $p_0(n;\ell)$ in the present paper, is the  positive root  of the quadratic equation
\begin{align} \label{exponent p0 He-Witt-Yin}
((\ell+1)n-1)p^2-((\ell+1)n+1-2\ell)p-2(\ell+1)=0.
\end{align}
For $\ell=0$ the exponent $p_0(n;\ell)$ coincides with  the so-called \emph{Strauss exponent} (denoted by $p_{\mathrm{Str}}(n)$), which is the critical exponent for the semilinear wave equation with power nonlinearity.

 This conjecture was made according to a blow-up result for the subcritical case $1<p<p_0(n;\ell)$ when $n\geqslant 2$ proved by means of the so-called Kato's lemma for second order ordinary differential inequalities. Then, in \cite{HWY16t} and \cite{HWY17d1}, the global existence of small data solutions in suitable weighted Sobolev spaces in space dimension $n\geqslant 3$ and $n=2$, respectively, is proved for $p>p_0(n;\ell)$ in the subconformal case. Furthermore, in \cite{HWY17d1} a blow-up result is proved in the critical case $p=p_0(n;\ell)$ as well (when $n\geqslant 2$). Recently, in \cite{HWY20} the one dimensional case has been considered and it has been shown that the critical exponent is no longer a generalized Strauss exponent, rather the Kato-type exponent $1+\frac{2}{\ell}$ (all explicit computations are done for the Tricomi operator, namely, for $\ell=1/2$).

On the other hand, for the semilinear generalized Tricomi equation with nonlinearity of derivative type 
\begin{align}\label{Semi_Tri_Non_Der}
\begin{cases} 
\partial_t^2 u- t^{2\ell}\Delta u =|\partial_t u|^p, &  x\in \mathbb{R}^n, \ t>0,\\
u(0,x)=\varepsilon u_0(x), & x\in \mathbb{R}^n, \\ 
\partial_t u(0,x)=\varepsilon u_1(x), & x\in \mathbb{R}^n,
\end{cases}
\end{align} in \cite{LP20} the second and the third author proved a blow-up result for $1<p\leqslant p_1(n;\ell)$, where $$ p_1(n;\ell) \doteq \frac{(\ell+1)n+1}{(\ell+1)n-1}=p_{\mathrm{Gla}}((\ell+1)n) \quad \mbox{and} \quad p_{\mathrm{Gla}}(n)\doteq  1+ \frac{2}{n-1}$$ is the so-called \emph{Glassey exponent}, which is, the critical exponent for the semilinear wave equation with nonlinearity of derivative type $|\partial_t u|^p$.  In \cite{LS20} the authors announced to have enlarged this blow-up range, employing an approach based on the test function method. 

Studying the blow-up dynamics for local solutions to the semilinear Cauchy problem \eqref{Semi_Tri_Combined}, we want to emphasize how the contemporary presence of a power nonlinearity $|u|^q$ and of a nonlinearity of derivative type $|\partial_t u|^p$ produces a larger blow-up region in the $(p,q)$ - plane than the one that has been got by dealing individually with each of these two kinds of nonlinear terms as in the treatment of  \eqref{Semi_Tri_Pow_Non} or  \eqref{Semi_Tri_Non_Der}.

Finally, we mention that the semilinear wave models with \emph{combined nonlinearity} $|\partial_t u|^p+|u|^q$ (namely, \eqref{Semi_Tri_Combined} for $\ell=0$) has been already studied in \cite{HZ14,HWY16,ISW18,DFW19}. In particular, in \cite{HZ14,ISW18} blow-up results are shown for $p,q>1$ such that $(q-1)((n-1)p-2)<4$, under suitable sign assumptions for the Cauchy data. On the other hand, in \cite{HWY16} it is shown the optimality of the previous condition, by proving the global existence of small data solution for $p,q>1$ such that $(q-1)((n-1)p-2)\geqslant 4$, $p>p_{\mathrm{Gla}}(n)$, and $q>p_{\mathrm{Str}}(n)$ in space dimension $n=2,3$. 

Finally, we point out that other semilinear wave models with combined nonlinearity have been investigated (from the viewpoint of the blow-up dynamics) in the \emph{scattering producing case} \cite{LT18ComNon} and in the \emph{scale-invariant case} \cite{HH20,HH20i,HH20m}.

\color{black}

\begin{definition} \label{Defn_Energy_Solution} Let $u_0\in H^1(\mathbb{R}^n)$, $u_1\in L^2(\mathbb{R}^n)$. We say that $u$ is an energy solution to \eqref{Semi_Tri_Combined} on $[0,T)$ if 
\begin{align*}
u\in \mathcal{C}\big([0,T), H^1(\mathbb{R}^n)\big) \cap \mathcal{C}^1\big([0,T), L^2(\mathbb{R}^n)
\big)\cap L^q_{\mathrm{loc}}((0,T)\times \mathbb{R}^n) \ \ \mbox{such that} \ \ \partial_t u \in  L^p_{\mathrm{loc}}((0,T)\times \mathbb{R}^n) 
\end{align*} satisfies $u(0,\cdot)= \varepsilon u_0$ in $H^1(\mathbb{R}^n)$ and the integral identity 
\begin{align}
& \int_{\mathbb{R}^n} \partial_t u(t,x) \psi(t,x) \, \mathrm{d}x -\int_0^t \int_{\mathbb{R}^n} \partial_t u(s,x) \psi_t(s,x) \, \mathrm{d}x \, \mathrm{d}s + \int_0^t \int_{\mathbb{R}^n} s^{2\ell}\, \nabla u(s,x) \cdot \nabla \psi(s,x) \, \mathrm{d}x \, \mathrm{d}s \notag \\
 &\quad =\varepsilon \int_{\mathbb{R}^n} u_1(x) \psi(0,x) \, \mathrm{d}x+ \int_0^t \int_{\mathbb{R}^n} \big(|\partial_t u(s,x)|^p+|u(s,x)|^q\big) \psi(s,x) \, \mathrm{d}x \, \mathrm{d}s \label{Defn_Eq_Energy_SOlution}
\end{align}
 for any test function $\psi \in \mathcal{C}^\infty_0([0,T)\times \mathbb{R}^n)$ and any $t\in (0,T)$.
\end{definition}
We point out that performing a further step of integration by parts in \eqref{Defn_Eq_Energy_SOlution}, we obtain the integral relation
\begin{align}
& \int_{\mathbb{R}^n} \left( \partial_t u(t,x) \psi(t,x)- u(t,x) \psi_t(t,x)\right) \mathrm{d}x + \int_0^t \int_{\mathbb{R}^n} u(s,x) \left( \psi_{tt}(s,x)- s^{2\ell}\Delta \psi(s,x)\right)  \mathrm{d}x \, \mathrm{d}s  \notag \\
 &\quad = \varepsilon \int_{\mathbb{R}^n} \big(u_1(x) \psi(0,x)-u_0(x) \psi_t(0,x)\big) \, \mathrm{d}x + \int_0^t \int_{\mathbb{R}^n} \big(|\partial_t u(s,x)|^p+|u(s,x)|^q\big) \psi(s,x) \, \mathrm{d}x \, \mathrm{d}s
\label{Defn_Eq_2_Energy_Solution}
\end{align} 
for any $\psi \in \mathcal{C}^\infty_0([0,T)\times \mathbb{R}^n)$ and any $t\in (0,T)$.


\begin{remark} Let us underline that we could work with less restrictive assumptions on the space for the solutions to \eqref{Semi_Tri_Combined}, requiring
 $u_0,u_1\in L^1_{\mathrm{loc}}(\mathbb{R}^n)$ with $\supp u_0$, $\supp u_1\subset B_R$  and assuming the existence of a solution
 \begin{align*}
u\in \mathcal{C}([0,T), W^{1,1}_{\mathrm{loc}}(\mathbb{R}^n)) \cap \mathcal{C}^1([0,T), L^1_{\mathrm{loc}}(\mathbb{R}^n))\cap L^q_{\mathrm{loc}}((0,T)\times \mathbb{R}^n) \ \ \mbox{such that} \ \ \partial_t u \in  L^p_{\mathrm{loc}}((0,T)\times \mathbb{R}^n) 
\end{align*}  
satisfying the support condition
\begin{align*} 
\supp u(t,\cdot) \subset B_{R+\, t^{\ell+1}/(\ell+1)} \qquad \mbox{for any} \ \ t\in (0,T),
\end{align*} 
 and  fulfilling the integral relations \eqref{Defn_Eq_Energy_SOlution} or, equivalently, \eqref{Defn_Eq_2_Energy_Solution} for any $\psi \in \mathcal{C}^\infty_0([0,T)\times \mathbb{R}^n)$ and any $t\in (0,T)$. However, we prefer to restrict our considerations to energy solutions for which the validity of both a local existence result for the Cauchy problem \eqref{Semi_Tri_Combined} and the property of finite propagation speed with the above written curved light-cone can be established by following the approach in \cite{DAn95}.
\end{remark}


Let us state the main result of this paper.
\begin{theorem} \label{Theorem}
Let $p,q>1$ satisfy
\begin{align} \label{blow up range p,q}
\big[((\ell+1)n-1)p-2\ell(p-1)-2\big](q-1)<4.
\end{align} Let us assume that $u_0\in H^1(\mathbb{R}^n)$ and $u_1\in L^2(\mathbb{R}^n)$ are nonnegative, not both trivial and compactly supported functions with supports contained in $B_R$ for some $R>0$. Let $$u\in \mathcal{C}([0,T), H^1(\mathbb{R}^n)) \cap \mathcal{C}^1([0,T), L^2(\mathbb{R}^n))\cap L^q_{\mathrm{loc}}((0,T)\times \mathbb{R}^n) \ \ \mbox{such that} \ \ \partial_t u \in  L^p_{\mathrm{loc}}((0,T)\times \mathbb{R}^n) $$ be an energy solution to \eqref{Semi_Tri_Combined} according to Definition \ref{Defn_Energy_Solution} with lifespan $T=T(\varepsilon)$.

Then, there exists a positive constant $\varepsilon_0=\varepsilon_0(n,\ell,p,q,u_0,u_1,R)$ such that for any $\varepsilon\in (0,\varepsilon_0]$ the energy solution $u$ blows up in finite time. Furthermore, the upper bound estimate for the lifespan
\begin{align*}
T(\varepsilon)\leqslant C \varepsilon^{-\frac{p(q-1)}{\theta(n,\ell,p,q)}}
\end{align*}
holds, where the positive constant $C$ is independent of $	\varepsilon$ and
\begin{align} \label{def theta}
\theta(n,\ell,p,q)\doteq 2 -\tfrac 12 \big[((\ell+1)n-1)p-2\ell(p-1)-2\big](q-1).
\end{align}
\end{theorem}

The paper is organized as follows: in Section \ref{Section proof main thm} we provide the proof of Theorem \ref{Theorem}. In Section \ref{Section |u|^q} we recall some results already known in the case of the semilinear generalized Tricomi equation with power nonlinearity that follow straightforwardly by slightly modifying the proof in Section \ref{Section proof main thm}. Finally, in Section \ref{Section Final Rem} we explain in detail the obtained blow-up range and we compare it with corresponding blow-up results for \eqref{Semi_Tri_Pow_Non} and \eqref{Semi_Tri_Non_Der}, respectively; in particular, we put a special emphasis on the analysis of the one dimensional case.

\subsection{Notations}

Throughout the paper we employ the following notations: 
the ball in $\mathbb{R}^n$ with radius $R$ around the origin is denoted $B_R$; the notations $f\lesssim g$ means that there exists a positive constant $C$ such that $f\leqslant C g$ and, similarly, for $f\gtrsim g$;  by writing $\psi\in \mathcal{C}^\infty_0([0,T)\times \mathbb{R}^n)$ we mean the existence of a compact set  $\mathcal{K}\subset \mathbb{R}^{1+n}$ such that the support of the $\mathcal{C}^\infty$ function $\psi$ satisfies $\mathrm{supp} \psi \subset \mathcal{K}\subset [0,T)\times \mathbb{R}^n$;
$\mathrm{K}_\nu$ denotes the modified Bessel function of second kind of order $\nu$. Finally, as in the introduction, $p_0(n;\ell)$ is the positive solution to \eqref{exponent p0 He-Witt-Yin}, $p_{\mathrm{Str}}(n)=p_0(n;0)$ denotes the Strauss exponent, and $p_{\mathrm{Gla}}(n)=p_1(n;0)$ denotes the Glassey exponent. When $1+2/(n-1)$ refers to the nonlinear term $|u|^q$ we call it \emph{Kato exponent} (denoted by $p_{\mathrm{Kat}}(n)$), as it is customary in the related literature, although it coincides with the Glassey exponent.

\section{Proof of the main result} \label{Section proof main thm}

Let $u$ be an energy solution to \eqref{Semi_Tri_Combined} according to Definition \ref{Defn_Energy_Solution} that fulfills the support condition
\begin{align}\label{finite speed propagation}
\mathrm{supp} \, u(t,\cdot) \subset B_{R+\phi_\ell(t)} \qquad \mbox{for any} \ t\in (0,T), \ \ \mbox{where} \ \  \phi_\ell(t) \doteq \tfrac{t^{\ell+1}}{\ell+1}.
\end{align}
  As time-dependent function, whose dynamic will provide the blow-up result, we consider the space average of $u$ (following an 
approach introduced for the first time in \cite{Kato80}), namely,
\begin{align*}
U(t)\doteq \int_{\mathbb{R}^n} u(t,x) \, \mathrm{d}x.
\end{align*}
We will derive the iteration frame for this functional. Furthermore, we will employ also two auxiliary functionals, whose definitions will be provided in the next subsection, see \eqref{Def_U0} and \eqref{Def_U1}, to determine a first lower bound estimate for $U$.

 
 \color{black}

\subsection{The test function}

In this section, we recall from \cite{HWY17} the definition of a function $\Psi=\Psi(t,x)$ with separate variables that solves the linear generalized Tricomi equation, namely 
\begin{align}\label{Adj_Eq}
\Psi_{tt}- t^{2\ell}\Delta \Psi=0.
\end{align}
As $x$-dependent function we consider the eigenfunction for the Laplacian introduced in \cite{YZ06}
\begin{align} \label{Eigen_Laplace}
\varphi(x)\doteq \begin{cases}   
\mathrm{e}^{x}+\mathrm{e}^{-x} & \mbox{if} \ n=1, \\
\displaystyle{\int_{\mathbb{S}^{n-1}}\mathrm{e}^{x\cdot \omega} \mathrm{d} \sigma_\omega} & \mbox{if} \ n\geqslant 2.
\end{cases}
\end{align} This function is radially symmetric,  belongs to the class $\mathcal{C}^\infty(\mathbb{R}^n)$ and satisfies the following notable properties:
\begin{align}
&\Delta \varphi = \varphi, \label{eigenfunction property} \\
& \varphi(x) \sim |x|^{-\frac{n-1}{2}}\mathrm{e}^{|x|} \quad \mbox{as} \ |x|\to \infty. \label{eigenfunction funct asy est}
\end{align}
In the next lines we employ some well-known properties of the modified Bessel function of the second kind. Nevertheless, for the sake of readability we address the reader to Appendix \ref{App A} for a short recap of these properties.  Let us recall the time-dependent function $\lambda=\lambda(t;\ell)$ introduced in \cite{HWY17} such that 
\begin{align} \label{Lambda_Fun}
\lambda(t;\ell)\doteq C_{\ell}\,t^{\frac{1}{2}} \mathrm{K}_{\frac{1}{2\ell+2}}\left(\tfrac{1}{\ell+1}t^{\ell+1}\right),
\end{align} 
with the positive constant $C_{\ell}$ allowing $\lambda(0;\ell)=1$, where $\mathrm{K}_{\nu}$ denotes the modified Bessel function of the second kind of order $\nu$. Note that we use \eqref{asymptotic K alpha small} to get a finite, positive value of $\lambda$ as $t\to 0$.
For the sake of brevity, hereafter, we will write simply $\lambda=\lambda(t)$ skipping the dependence on $\ell$. Due to the fact that $\mathrm{K}_{\nu}$ fulfills the second order ordinary differential equation \eqref{ODE Mod Bessel funct 2nd kind}
 by straightforward computations it follows that $\lambda$ satisfies the following problem:
\begin{align}\label{ODE_Lambda}
\begin{cases}
\lambda''(t)-t^{2\ell}\lambda(t)=0,&t>0,\\
\lambda(0)=1,\ \lambda(\infty)=0.
\end{cases}
\end{align} 
Therefore, combining \eqref{eigenfunction property} and \eqref{ODE_Lambda}, it follows that the positive function $$\Psi(t,x)\doteq \lambda(t)\varphi(x)$$ solves the desired 
linear equation \eqref{Adj_Eq}. 

Let us point out that, using the recursive relations for the derivatives of modified Bessel functions of the second kind (cf. \eqref{recursive identity K' alpha} in Appendix \ref{App A})
we have
\begin{align}\label{lambda_prop_01}
\lambda'(t)&=\tfrac{C_{\ell}}{2}t^{-\frac{1}{2}}\mm{K}_{\frac{1}{2\ell+2}}\left(\tfrac{1}{\ell+1}t^{\ell+1}\right)+C_{\ell}\,t^{\frac{1}{2}}\left(\mm{K}_{\frac{1}{2\ell+2}}\left(\tfrac{1}{\ell+1}t^{\ell+1}\right)\right)'\notag\\
&=-C_{\ell}\,t^{\frac{1}{2}+\ell}\mm{K}_{\frac{1}{2\ell+2}-1}\left(\tfrac{1}{\ell+1}t^{\ell+1}\right)<0,
\end{align} where in the last line we used that $\mathrm{K}_\nu(\tau)$ is a positive function for $\nu\in \mathbb{R}$ and $\tau>0$.
Furthermore, there exists a constant $c_0>1$ such that
\begin{align}\label{lambda_prop_02}
\frac{|\lambda'(t)|}{\lambda(t)}\geqslant\frac{t^{\ell}}{c_0}\ \ \mbox{for any}\ \ t >0  \quad \mbox{and}  \quad  \frac{|\lambda'(t)|}{\lambda(t)}\leqslant c_0 \,  t^{\ell} \ \ \mbox{for any}\ \ t \geqslant 1,
\end{align}
as it has been shown in \cite[Lemma 2.1]{HL96}. Note that in \eqref{lambda_prop_02} we used that $\lambda$ is positive function ($\lambda$ is strictly decreasing and $\lambda(\infty)=0$).

Finally, we can introduce the two following auxiliary functionals:
\begin{align}
U_0(t) & \doteq \int_{\mathbb{R}^n}  u(t,x) \Psi(t,x) \, \mathrm{d}x, \label{Def_U0}\\
U_1(t) & \doteq \int_{\mathbb{R}^n} \partial_t u(t,x) \Psi(t,x) \, \mathrm{d}x. \label{Def_U1}
\end{align} Our strategy in the next subsection is to determine  lower bound estimates for $U_0,U_1$.

\subsection{Lower bound estimates for the auxiliary functionals}

Let us start this subsection by deriving a lower bound estimate for $U_0$.

\begin{lemma} \label{Lemma U0} Let $u_0,u_1$ be functions satisfying the same assumptions as in the statement of Theorem \ref{Theorem}. Let $U_0$ be the functional associated to a local (in time) energy solution $u$ of \eqref{Semi_Tri_Combined} and defined by \eqref{Def_U0}. Then, there exists $T_0> 0$ such that 
\begin{align}\label{lower bound U0}
U_0(t) \gtrsim \varepsilon I_{\ell}[u_0,u_1]\,t^{-\ell}\ \ \mbox{for any}\  t\in [2 T_0,T),
\end{align} where 
\begin{align} \label{def I[u0,u1]}
 I_{\ell}[u_0,u_1]\doteq \int_{\mathbb{R}^n} \big( u_1(x)-\lambda'(0;\ell)u_0(x) \big) \varphi(x) \, \mathrm{d}x  .
\end{align}
\end{lemma}
\begin{proof}
Due to the property of finite speed of propagation, we have that $u$ fulfills \eqref{finite speed propagation}, where $R>0$ is chosen so that $\supp u_0$, $\supp u_1\subset B_R$.
In particular, thanks to this support condition for the solution $u$ we may consider not compactly supported test functions in \eqref{Defn_Eq_Energy_SOlution}. Therefore, choosing $\Psi=\lambda\varphi\in \mathcal{C}^\infty([0,T)\times \mathbb{R}^n)$ as test function in \eqref{Defn_Eq_2_Energy_Solution}, for any $t\in(0,T)$ we get
\begin{align*}
& \int_{\mathbb{R}^n}  \partial_t u(t,x) \Psi(t,x) \, \mathrm{d}x - \int_{\mathbb{R}^n}   u(t,x) \Psi_t(t,x) \, \mathrm{d}x \\ 
& \quad =\varepsilon \int_{\mathbb{R}^n} \big( u_1(x)-\lambda'(0)u_0(x) \big) \varphi(x) \, \mathrm{d}x   +  \int_0^t \int_{\mathbb{R}^n} \big(|\partial_t u(s,x)|^p+|u(s,x)|^q\big) \Psi(s,x) \, \mathrm{d}x \, \mathrm{d}s \\ 
& \quad = \varepsilon I_{\ell}[u_0,u_1] +  \int_0^t \int_{\mathbb{R}^n} \big(|\partial_t u(s,x)|^p+|u(s,x)|^q\big) \Psi(s,x) \, \mathrm{d}x \, \mathrm{d}s,  
\end{align*} where we employed $\lambda(0)=1$ and \eqref{Adj_Eq}. Hence, by using \eqref{Def_U0}, we may rewrite the previous identity as follows:
\begin{align}\label{ODE_U0}
U'_0(t) -2 \frac{\lambda'(t)}{\lambda(t)} U_0(t)  = \varepsilon I_{\ell}[u_0,u_1] +  \int_0^t \int_{\mathbb{R}^n} \big(|\partial_t u(s,x)|^p+|u(s,x)|^q\big) \Psi(s,x) \, \mathrm{d}x \, \mathrm{d}s 
\end{align} 
for any $t\in(0,T)$. Since the nonlinearity is nonnegative, from \eqref{ODE_U0} we have
\begin{align}\label{First_Order_Ineq_U'0}
U'_0(t) -2 \frac{\lambda'(t)}{\lambda(t)} U_0(t) \geqslant  \varepsilon I_{\ell}[u_0,u_1]
\end{align}  
for any $t\in(0,T)$. Multiplying both sides of the previous inequality by $(\lambda(t))^{-2}$ and integrating over $[0,t]$, it results
\begin{align} \label{intermediate estimate U0}
\frac{U_0(t)}{\lambda^2(t)}- \frac{U_0(0)}{\lambda^2(0)} \geqslant \varepsilon I_{\ell}[u_0,u_1] \int_0^t\frac{\mathrm{d}s}{\lambda^2(s)}.
\end{align} Thus, employing the sign assumption on the first initial data, we find
\begin{align*}
U_0(t) \geqslant \varepsilon I_{\ell}[u_0,u_1] \, \lambda^2(t) \int_0^t\frac{\mathrm{d}s}{\lambda^2(s)}.
\end{align*} 
By using the asymptotic behavior for the modified Bessel function of the second kind (see \eqref{asymptotic K alpha} in Appendix \ref{App A}), we have
\begin{align} \label{asymptotic behavior lambda}
\lambda(t) = \sqrt{\frac{\pi(\ell+1)}{2}} \, C_{\ell} \, t^{-\frac{\ell}{2}}\mathrm{e}^{-\phi_\ell(t)}\left(1+O\big(t^{-(\ell+1)}\big)\right)\ \  \mbox{as} \ \ t\to \infty.
\end{align} Consequently, for large $t$ so that $0<T_0\leqslant t<T$, we can estimate
\begin{align*}
U_0(t) & \gtrsim  \varepsilon I_{\ell}[u_0,u_1]\,t^{-\ell}\mathrm{e}^{-2 \phi_\ell(t)} \int_{T_0}^ts^{\ell}\mathrm{e}^{2\phi_\ell(s)} \, \mathrm{d}s.
\end{align*}
Therefore, for $t\in [2 T_0,T)$
 \begin{align*} 
U_0(t) & \gtrsim  \varepsilon I_{\ell}[u_0,u_1]\,t^{-\ell}\mathrm{e}^{-2 \phi_\ell(t)} \int_{t/2}^ts^{\ell}\mathrm{e}^{2\phi_\ell(s)} \, \mathrm{d}s\\
&  \gtrsim  \varepsilon I_{\ell}[u_0,u_1]\,t^{-\ell} \left(1-\mathrm{e}^{2\phi_\ell(t/2) -2\phi_\ell(t)}\right) = \varepsilon I_{\ell}[u_0,u_1]\,t^{-\ell} \left(1-\mathrm{e}^{2(2^{-(\ell+1)}-1)\phi_\ell(t)}\right)\\
& \gtrsim  \varepsilon I_{\ell}[u_0,u_1]\,t^{-\ell}.
\end{align*}
 This completes the proof.
\end{proof}

Next we derive a lower bound estimate for the functional $U_1$. The proof of the next result is inspired by \cite[Lemma 3.3]{HH20i}.
\begin{lemma} \label{Lemma U1} Let $u_0,u_1$ be functions satisfying the same assumptions as in the statement of Theorem \ref{Theorem}. Let $U_1$ be the functional associated to a local (in time) energy solution $u$ of \eqref{Semi_Tri_Combined} and defined by \eqref{Def_U1}. Then, there exists $T_0>0$ such that 
\begin{align} \label{lower bound U1}
U_1(t) \gtrsim\varepsilon I_{\ell}[u_0,u_1] \ \  \mbox{for any} \  t\in [2 T_0,T),
\end{align}
where $I_{\ell}[u_0,u_1]$ was defined in \eqref{def I[u0,u1]}.
\end{lemma}

\begin{proof} 
Let us get started by showing that $U_1(t)\geqslant 0$ for any $t\in [0,T)$. According to this purpose, we introduce the further auxiliary functional  
\begin{align*}
F_1(t)\doteq \int_{\mathbb{R}^n} \partial_t u(t,x) \varphi(x) \, \mathrm{d}x \quad \mbox{for} \  t\in [0,T).
\end{align*} Since $U_1=\lambda F_1$ and $\lambda$ is a positive function, if we show that $F_1$ is nonnegative, then, it follows the nonnegativity of $U_1$ as well. Choosing $\varphi$ as test function in \eqref{Defn_Eq_2_Energy_Solution} (this is possible, due to the support condition for $u$ as explained in the proof of Lemma \ref{Lemma U0}), 
 we arrive at
\begin{align*}
 \int_{\mathbb{R}^n} \partial_t u(t,x) \varphi(x) \, \mathrm{d}x & =  \int_0^t \int_{\mathbb{R}^n} s^{2\ell} u(s,x) \Delta\varphi(x) \, \mathrm{d}x\, \mathrm{d}s+\varepsilon \int_{\mathbb{R}^n}u_1(x) \varphi(x) \, \mathrm{d}x+ \int_0^t \int_{\mathbb{R}^n} \left(|\partial_t u(s,x)|^p+|u(s,x)|^q\right)\varphi(x) \, \mathrm{d}x\, \mathrm{d}s \\
 & = \int_0^t  \frac{s^{2\ell}}{ \lambda(s)} U_0(s) \, \mathrm{d}s+\varepsilon \int_{\mathbb{R}^n}u_1(x) \varphi(x) \, \mathrm{d}x+ \int_0^t \int_{\mathbb{R}^n} \left(|\partial_t u(s,x)|^p+|u(s,x)|^q\right)\varphi(x) \, \mathrm{d}x\, \mathrm{d}s.
\end{align*} We remark that \eqref{intermediate estimate U0} implies that $U_0(t)$ is nonnegative for $t \in [0,T)$. Therefore, due to the fact that the second data and the nonlinearity are nonnegative, from the previous identity, it follows that $F_1$ is nonnegative. We prove now the lower bound estimate \eqref{lower bound U1}.
From \eqref{Def_U0} and \eqref{Def_U1}, it follows the relations
\begin{align}\label{relation U0 U1}
U_1(t) = U_0'(t)-\frac{\lambda'(t)}{\lambda(t)}U_0(t).
\end{align} 
Therefore, from \eqref{ODE_U0} we obtain
\begin{align} \label{control U1 + (..) U_0}
U_1(t) - \frac{\lambda'(t)}{\lambda(t)} U_0(t)= \varepsilon I_{\ell}[u_0,u_1] +  \int_0^t \int_{\mathbb{R}^n} \big(|\partial_t u(s,x)|^p+|u(s,x)|^q\big) \Psi(s,x) \, \mathrm{d}x \, \mathrm{d}s 
\end{align} 
for any $t\in(0,T)$. Differentiating the last identity with respect to $t$, we have
\begin{align*}
U_1'(t)- \frac{\lambda'(t)}{\lambda(t)} U'_0(t) +\left(- \frac{\lambda''(t)}{\lambda(t)} + \left(\frac{\lambda'(t)}{\lambda(t)}\right)^2 \right) U_0(t)  = \int_{\mathbb{R}^n} \big(|\partial_t u(t,x)|^p+|u(t,x)|^q\big) \Psi(t,x) \, \mathrm{d}x 
\end{align*} 
for any $t\in(0,T)$. The employment of \eqref{relation U0 U1} and \eqref{ODE_Lambda} in the last equation yields
\begin{align*}
\int_{\mathbb{R}^n} \big(|\partial_t u(t,x)|^p+|u(t,x)|^q\big) \Psi(t,x) \, \mathrm{d}x 
& = U_1'(t)-\frac{\lambda'(t)}{\lambda(t)}U_1(t)-\frac{\lambda''(t)}{\lambda(t)}U_0(t)\\
&   = U_1'(t)-\frac{\lambda'(t)}{\lambda(t)}U_1(t)-t^{2\ell}U_0(t).
\end{align*} Using again the nonnegativity of the nonlinearity, for a fixed parameter $\omega$ 
 we may rewrite
\begin{align}
0 & \leqslant U_1'(t) - \frac{\lambda'(t)}{\lambda(t)} U_1(t)-t^{2\ell} U_0(t) \notag \\
&=U_1'(t)-2\omega\frac{\lambda'(t)}{\lambda(t)}U_1(t)-h_1(t)\left(U_1(t)-\frac{\lambda'(t)}{\lambda(t)}U_0(t)\right)-h_2(t)U_0(t),
 \label{ODI mix U1 U0}
\end{align}
 where the functions $h_1,h_2$ are defined as follows:
\begin{align*}
h_1(t) & \doteq \frac{\lambda'(t)}{\lambda(t)}(1-2\omega),  \quad  
h_2(t)  \doteq t^{2\ell}+\left(\frac{\lambda'(t)}{\lambda(t)}\right)^2(1-2\omega). 
\end{align*}
By using \eqref{lambda_prop_01} and \eqref{lambda_prop_02} and the choosing  $\omega \in \left(\frac12, \frac12+\frac {1}{2c_0^2}\right)$, we obtain that 
\begin{align}\label{h1,h2 low bou est}
h_1(t) & \geqslant \frac{2\omega-1}{c_0} \, t^\ell  \geqslant 0 \quad \mbox{and} \quad
h_2(t)  \geqslant (1+c_0^2(1-2\omega))  \, t^{2\ell} \geqslant 0
\end{align}
for any $t\geqslant 1$.


Therefore, combining 
\eqref{control U1 + (..) U_0}, \eqref{ODI mix U1 U0} and \eqref{h1,h2 low bou est}, we get
\begin{align*}
U_1'(t)-2\omega\frac{\lambda'(t)}{\lambda(t)}U_1(t) & \geqslant h_1(t)\left(U_1(t)-\frac{\lambda'(t)}{\lambda(t)}U_0(t)\right)+h_2(t)U_0(t) \\
 & \geqslant h_1(t)\left(U_1(t)-\frac{\lambda'(t)}{\lambda(t)}U_0(t)\right) \\
 & \gtrsim   \varepsilon I_{\ell}[u_0,u_1] \, t^\ell   
\end{align*} for $t\in (1,T)$. We remark that in the second step of the previous chain of inequalities the nonnegativity of the functional $U_0$ is employed. Multiplying the last estimate by $ (\lambda(t))^{-2\omega}$ and integrating over $[1,t]$, we arrive at
\begin{align} \label{est inter U1}
\frac{U_1(t)}{\lambda^{2\omega}(t)} -\frac{U_1(1)}{\lambda^{2\omega}(1)}  \gtrsim  \varepsilon  I_{\ell}[u_0,u_1]\int_{1}^t\frac{s^{\ell}}{\lambda^{2\omega}(s)}\,\mathrm{d}s.
\end{align} 
Employing the nonnegativity of the functional $U_1$, that we proved at the really beginning of this proof, we may neglect the second term on the left-hand side of \eqref{est inter U1}.
Next, using again the asymptotic behavior of $\lambda$ given in \eqref{asymptotic behavior lambda} as in the proof of Lemma \ref{Lemma U0}, we derive the lower bound estimate for $U_1$. From \eqref{est inter U1} it follows 
\begin{align*}
 U_1(t) &\gtrsim \varepsilon I_{\ell}[u_0,u_1]\lambda^{2\omega}(t)\int_{1}^t\frac{s^{\ell}}{\lambda^{2\omega}(s)}\,\mathrm{d}s\\
 &\gtrsim\varepsilon I_{\ell}[u_0,u_1]\,t^{-\ell\omega}\mathrm{e}^{-2\omega \phi_\ell(t)}\int_{T_0}^ts^{\ell+\ell\omega}\mathrm{e}^{2\omega\phi_\ell(s)}\,\mathrm{d}s\\
 &\gtrsim\varepsilon I_{\ell}[u_0,u_1]\, \mathrm{e}^{-2\omega \phi_\ell(t)}\int_{t/2}^ts^{\ell}\mathrm{e}^{2\omega\phi_\ell(s)}\,\mathrm{d}s \\
    &\gtrsim \varepsilon I_{\ell}[u_0,u_1]\left(1-\mathrm{e}^{-2\omega (1-2^{-(\ell+1)})\phi_\ell(t)}\right)\\
 &\gtrsim \varepsilon I_{\ell}[u_0,u_1]
\end{align*} 
for any $t \in [2T_0,T)$. Note that in the previous steps we may assume $T_0>1$ without loss of generality.  The proof is complete.

\end{proof}

\subsection{Derivation of the iteration frame} \label{Subsection der iter frame}

%

Let us choose a bump function which is equal to $1$ on the light-cone $\{(s,x)\in[0,t]\times\mb{R}^n:|x|\leqslant R+\phi_\ell(s) \}$. Then, from \eqref{Defn_Eq_Energy_SOlution} we obtain
\begin{align*}
\int_{\mb{R}^n}\partial_tu(t,x)\,\mathrm{d}x=\varepsilon\int_{\mb{R}^n}u_1(x)\,\mathrm{d}x+\int_0^t\int_{\mb{R}^n}\big(|\partial_tu(s,x)|^p+|u(s,x)|^q\big)\,\mathrm{d}x\,\mathrm{d}s,
\end{align*}
which leads to
\begin{align}
U(t)&=U(0)+U'(0)t+\int_0^t\int_0^s\int_{\mb{R}^n}\big(|\partial_tu(\tau,x)|^p+|u(\tau,x)|^q\big)\,\mathrm{d}x\,\mathrm{d}\tau\,\mathrm{d}s \label{Eq_U}\\
&\geqslant \int_0^t\int_0^s\int_{\mb{R}^n}\big(|\partial_tu(\tau,x)|^p+|u(\tau,x)|^q\big)\,\mathrm{d}x\,\mathrm{d}\tau\,\mathrm{d}s. \label{Eq_Pre_Frame}
\end{align}
Additionally, from nonnegativity of $u_0,u_1$ it follows that $U(t)\geqslant0$ for any $t\in (0,T)$.

 From the paper \cite{HWY17}, we know 
 that
\begin{align*}
\left(\int_{B_{R+\phi_\ell(\tau)}}\Psi^{\frac{p}{p-1}}(\tau,x)\,\mathrm{d}x\right)^{p-1}\lesssim \tau^{-\frac{\ell p}{2}}(R+\phi_\ell(\tau))^{(n-1)(p-1)-\frac{n-1}{2}p},
\end{align*} where hereafter the unexpressed multiplicative constants may depend on $n,p,R,\ell,u_0,u_1$ but are independent of $\varepsilon$.
By using H\"older's inequality and Lemma \ref{Lemma U1}, we obtain
\begin{align}\label{Est_Inte_Dtu}
\int_{\mb{R}^n}|\partial_tu(\tau,x)|^p\,\mathrm{d}x&\geqslant |U_1(\tau)|^p\left(\int_{B_{R+\phi_\ell(\tau)}}\Psi^{\frac{p}{p-1}}(\tau,x)\,\mathrm{d}x\right)^{-(p-1)}\notag\\
&\gtrsim \varepsilon^p\tau^{\frac{\ell p}{2}}(R+\phi_\ell(\tau))^{(n-1)(1-\frac{p}{2})}
\end{align}
for any $t\in (2T_0,T)$. Therefore, if we combine \eqref{Eq_Pre_Frame} and \eqref{Est_Inte_Dtu}, we derive a first lower bound for $U$, namely,
\begin{align}\label{First_Lower_Bound_1}
U(t)& \gtrsim \varepsilon^p(R+\phi_\ell(t))^{-(n-1)\frac{p}{2}}\int_{2T_0}^t\int_{2T_0}^s\tau^{\frac{\ell p}{2}}(R+\phi_{\ell}(\tau))^{n-1}\,\mathrm{d}\tau\,\mathrm{d}s\notag\\
&  \gtrsim \varepsilon^p (1+t)^{-(\ell+1)(n-1)\frac{p}{2}}\int_{2T_0}^t\int_{2T_0}^s(\tau-2T_0)^{\frac{\ell p}{2}+(\ell+1)(n-1)}\,\mathrm{d}\tau\,\mathrm{d}s\notag\\
&\geqslant K\varepsilon^p  (1+t)^{-(\ell+1)(n-1)\frac{p}{2}} (t-2T_0)^{\frac{\ell p}{2}+(\ell+1)(n-1)+2}
\end{align}
for any $t\in (2T_0,T)$, where $K$ is a suitable positive constant.


 Eventually, in order to construct an iteration frame, we apply H\"older's inequality  together with the property of finite speed of propagation. More precisely, we  find
\begin{align}
U(t) &\gtrsim \int_0^t\int_0^s(R+\phi_{\ell}(\tau))^{-n(q-1)}(U(\tau))^q\,\mathrm{d}\tau\,\mathrm{d}s \notag \\
&\geqslant C \int_0^t\int_0^s(1+\tau)^{-(\ell+1)n(q-1)}(U(\tau))^q\,\mathrm{d}\tau\,\mathrm{d}s 
\label{Iteration_Frame}
\end{align}
for any $t\in (0,T)$ and for a suitable positive constant $C$,  where we neglected the influence of the nonlinearity of derivative type in \eqref{Eq_Pre_Frame}. In the next section, we will employ \eqref{Iteration_Frame} to get iteratively a sequence of lower bound estimates for the functional $U(t)$ starting from \eqref{First_Lower_Bound_1}.

\begin{remark} \label{Rem lb} Repeating similar computations to those we made to get \eqref{First_Lower_Bound_1} (in particular, by using the lower bound estimate for $U_0$ from Lemma \ref{Lemma U0} instead of the lower bound estimate for $U_1$ from Lemma \ref{Lemma U1}), we obtain
\begin{align}\label{First_Lower_Bound_2}
U(t)&\gtrsim \varepsilon^q(R+\phi_{\ell}(t))^{-\frac{(n-1)q}{2}}\int_{2T_0}^t\int_{2T_0}^s\tau^{-\frac{\ell q}{2}}(R+\phi_\ell(\tau))^{n-1}\,\mathrm{d}\tau\,\mathrm{d}s\notag\\
&\geqslant \widetilde{K}\varepsilon^q(1+t)^{-(n-1)(\ell+1)\frac{q}{2}-\frac{\ell q}{2}}(t-2T_0)^{(n-1)(\ell+1)+2}
\end{align} for any $t\in (2T_0,T)$, where $\widetilde{K}$ is a suitable positive constant.
By using either \eqref{First_Lower_Bound_1} or \eqref{First_Lower_Bound_2}, we employed lower bound estimates for the nonlinearity of derivative type or for the power nonlinearity to get a first estimate from below for the functional $U$. 

A further lower bound bound estimate for $U$ can be obtained by assuming $u_1$ nontrivial (so that, $\int_{\mathbb{R}^n}u_1(x) \, \mathrm{d}x>0$). Under this additional assumption, from \eqref{Eq_U} we derive
\begin{align}\label{First_Lower_Bound_3}
U(t)\geqslant \widetilde{K}_0 \varepsilon t,
\end{align} for any $t\in (0,T)$, where the positive multiplicative constant $\widetilde{K}_0$ depends on $u_1$.
\end{remark}

\subsection{Iteration argument} \label{Subsection iter arg}

In this section, we are going to prove a sequence of lower bound estimates for the time-dependent functional $U$. In fact, we will prove that
\begin{align} \label{U(t) low boun j}
U(t)\geqslant C_j (1+t)^{-\alpha_j}(t-2T_0)^{\beta_j}
\end{align} for any $t\in (2T_0,T)$ and any $j\in\mathbb{N}$, where $\{\alpha_j\}_{j\in\mathbb{N}}$, $\{\beta_j\}_{j\in\mathbb{N}}$ and $\{C_j\}_{j\in\mathbb{N}}$ are sequences of nonnegative real numbers that will be determined recursively in the iteration step.

 Clearly, \eqref{U(t) low boun j} holds true for $j=0$ thanks to \eqref{First_Lower_Bound_1}, provided that $$C_0\doteq K\varepsilon^p, \quad \alpha_0\doteq (\ell+1)(n-1) \tfrac{ p}{2}, \quad \beta_0=\tfrac{\ell p}{2}+(\ell+1)(n-1)+2.$$
 
In order to use an inductive argument, it remains to show the validity of the inductive step: we assume the validity of \eqref{U(t) low boun j} for some $j\in \mathbb{N}$ and we have to prove its validity for $j+1$ too. Plugging \eqref{U(t) low boun j} in \eqref{Iteration_Frame}, we get
\begin{align*}
U(t) & \geqslant C C_j^q \int_{2T_0}^t\int_{2T_0}^s(1+\tau)^{-(\ell+1)n(q-1)-\alpha_j q}(\tau-2T_0)^{\beta_j q}\,\mathrm{d}\tau\,\mathrm{d}s \\
 & \geqslant C C_j^q (1+t)^{-(\ell+1)n(q-1)-\alpha_j q} \int_{2T_0}^t\int_{2T_0}^s (\tau-2T_0)^{\beta_j q}\,\mathrm{d}\tau\,\mathrm{d}s \\
 & \geqslant C C_j^q (\beta_j q+1)^{-1} (\beta_j q+2)^{-1} (1+t)^{-(\ell+1)n(q-1)-\alpha_j q} (t-2T_0)^{\beta_j q+2}.
\end{align*} Thus, we proved \eqref{U(t) low boun j} for $j+1$ provided that
\begin{align}
C_{j+1} & \doteq   C C_j^q (\beta_j q+1)^{-1} (\beta_j q+2)^{-1}, \label{def Cj} \\
\alpha_{j+1} & \doteq \underbrace{(\ell+1)n(q-1)}_{\doteq a}+\alpha_j q, \quad \beta_{j+1}\doteq \beta_j q+2. \label{def rec alpha j, beta j}
\end{align} In particular, from \eqref{def rec alpha j, beta j} we have
\begin{align}
\alpha_j &= a+q \alpha_{j-1} = \cdots = a \sum_{k=0}^{j-1} q^k +\alpha_0 q^j = \left(\tfrac{a}{q-1}+\alpha_0\right)q^j-\tfrac{a}{q-1}, \label{def alpha j}\\
\beta_j &= 2+q \beta_{j-1} = \cdots = 2 \sum_{k=0}^{j-1} q^k +\beta_0 q^j = \left(\tfrac{2}{q-1}+\beta_0\right)q^j-\tfrac{2}{q-1}. \label{def beta j}
\end{align} Therefore, combining \eqref{def rec alpha j, beta j} and \eqref{def beta j}, it results
\begin{align*}
\beta_j= 2+q \beta_{j-1}< \left(\tfrac{2}{q-1}+\beta_0\right)q^j.
\end{align*} Consequently, from \eqref{def Cj} it follows
\begin{align*}
C_{j} &  \geqslant   C C_{j-1}^q (\beta_{j-1} q+2)^{-2} \geqslant  \underbrace{C  \left(\tfrac{2}{q-1}+\beta_0\right)^{-2}}_{\doteq \widetilde{C}} C_{j-1}^q q^{-2j}.
\end{align*} Applying the logarithmic function to both sides of the last inequality and using in an iterative way the resulting relation, we obtain
\begin{align*}
\log C_j &\geqslant q\log C_{j-1}-2j \log q + \log \widetilde{C} \\
& \geqslant q^2\log C_{j-2}-2(j+(j-1)q)\log q+(1+q)\log \widetilde{C} \\
& \geqslant \cdots \geqslant q^{j}\log C_0-2\log q \,\sum_{k=0}^{j-1} (j-k)q^{k}+\log \widetilde{C} \, \sum_{k=0}^{j-1}q^k.
\end{align*}
Using the  formula
\begin{align*} 
\sum_{k=0}^{j-1}(j-k)q^{k}=\frac{1}{q-1}\bigg(\frac{q^{j+1}-q}{q-1}-j\bigg) 
\end{align*}
(whose validity can be proved through an inductive argument) we find
\begin{align*}
\log C_j &  \geqslant  q^{j}\log C_0- \frac{2\log q}{q-1}\bigg(\frac{q^{j+1}-q}{q-1}-j\bigg) +\frac{q^j-1 }{q-1}  \log \widetilde{C} \\ 
& = q^{j}\left(\log C_0 -\frac{2 q\log q}{(q-1)^2}+ \frac{ \log \widetilde{C} }{q-1}\right)+\frac{2 j\log q}{q-1}+\frac{2 q\log q}{(q-1)^2} -\frac{ \log \widetilde{C} }{q-1}.
\end{align*} If we denote by $j_0=j_0(n,p,q,\ell)\in \mathbb{N}$ the smallest integer greater than $\frac{\log\widetilde{C}}{2\log q}-\frac{q}{q-1}$, then, for any $j\geqslant j_0$ it holds
\begin{align}
\log C_j \geqslant  q^{j}\left(\log C_0 -\frac{2 q\log q}{(q-1)^2}+ \frac{ \log \widetilde{C} }{q-1}\right) = q^{j}\log \left( K q^{-(2 q)/(q-1)^2} \widetilde{C}^{1/(q-1)} \varepsilon^p\right) = q^{j}\log \left( D \varepsilon^p\right),
\label{lb log Cj} 
\end{align} where $D\doteq K q^{-(2 q)/(q-1)^2} \widetilde{C}^{1/(q-1)}$. Combining \eqref{U(t) low boun j}, \eqref{def alpha j}, \eqref{def beta j}  and \eqref{lb log Cj}, for $j\geqslant j_0$ and $t\geqslant 2T_0$ we arrive at
\begin{align*}
U(t) & \geqslant  \exp \left(q^j\log(D\varepsilon^p)\right) (1+t)^{-\alpha_j}(t-2T_0)^{\beta_j} 
\\ &=\exp\left(q^j\left(\log \left( D \varepsilon^p\right)-\left(\tfrac{a}{q-1}+\alpha_0\right)\log (1+t)+\left(\tfrac{2}{q-1}+\beta_0\right)\log (t-2T_0)\right)\right) (1+t)^{\frac{a}{q-1}} (t-2T_0)^{-\frac{2}{q-1}}.
\end{align*}
For $t\geqslant T_1\doteq \max\{1,4T_0\}$, we can estimate $\log(1+t)\leqslant \log (2t)$ and $\log(t-2T_0)\geqslant \log(t/2)$, so for $j\geqslant j_0$ it results
\begin{align}
U(t) & \geqslant \exp\left(q^j\left(\log \left( D \varepsilon^p\right)+\left(\tfrac{2-a}{q-1}+\beta_0-	\alpha_0\right)\log t-\left(\tfrac{a+2}{q-1}+\alpha_0+\beta_0\right)\log 2\right)\right) (1+t)^{\frac{a}{q-1}} (t-2T_0)^{-\frac{2}{q-1}} \notag \\
 & = \exp\left(q^j\left(\log \left( 2^{-(\alpha_0+\beta_0)-\frac{a+2}{q-1}} D \varepsilon^p t^{\frac{\theta(n,\ell,p,q)}{q-1}}\right)\right)\right) (1+t)^{\frac{a}{q-1}} (t-2T_0)^{-\frac{2}{q-1}}, \label{final lb U subcrit}
\end{align} where for the exponent of $t$ in the last equality we used 
\begin{align}
\tfrac{2-a}{q-1}+\beta_0-\alpha_0 & = \tfrac{2}{q-1} -(\ell+1)n +(\ell+1)(n-1)+\tfrac{\ell p}{2}+2-(\ell+1)(n-1)\tfrac{p}{2}\notag \\
& = \tfrac{2}{q-1} -((\ell+1)n -1-2\ell) \tfrac p2 -\ell +1 \notag \\
& = \tfrac{2}{q-1} -\tfrac 12 \big[ ((\ell+1)n -1) p -2\ell (p-1) -2 \big]= \tfrac{\theta(n,\ell,p,q)}{q-1}, \label{exponent t subcrit}
\end{align} where  $\theta(n,\ell,p,q)$ is defined in \eqref{def theta}. Note that the blow-up condition on $p,q$ in \eqref{blow up range p,q} is equivalent to require $\theta(n,\ell,p,q)>0$.
Let us choose $\varepsilon_0>0$ satisfying
\begin{align*}
\varepsilon_0^{-\frac{p(q-1)}{\theta(n,\ell,p,q)}}\geqslant 2^{-\frac{(\alpha_0+\beta_0)(q-1)+a+2}{\theta(n,\ell,p,q)}} D^{\frac{q-1}{\theta(n,\ell,p,q)}}T_1.
\end{align*} So, for any $\varepsilon\in(0,\varepsilon_0]$ and for $t\geqslant 2^{\frac{(\alpha_0+\beta_0)(q-1)+a+2}{\theta(n,\ell,p,q)}} D^{-\frac{q-1}{\theta(n,\ell,p,q)}}\varepsilon^{-\frac{p(q-1)}{\theta(n,\ell,p,q)}}$ it results
\begin{align*}
t\geqslant T_1 \qquad \mbox{and} \qquad 2^{-(\alpha_0+\beta_0)-\frac{a+2}{q-1}} D \varepsilon^p t^{\frac{\theta(n,\ell,p,q)}{q-1}}>1,
\end{align*} hence, letting $j\to \infty$ in \eqref{final lb U subcrit} it turns out that $U(t)$ blows up in finite time. Therefore, we proved the blowing-up of $U$ for any $\varepsilon\in(0,\varepsilon_0]$ whenever \eqref{blow up range p,q} holds and, besides, as byproduct we established the upper bound estimate for the lifespan $T(\varepsilon)\lesssim\varepsilon^{-\frac{p(q-1)}{\theta(n,\ell,p,q)}} $.

\section{Lifespan estimates for the power nonlinearity} \label{Section |u|^q}

In Section \ref{Section proof main thm}, we proved Theorem \ref{Theorem} by using an iteration argument, whose first lower bound estimate for $U$ is given by \eqref{First_Lower_Bound_1}. In particular, we combined the influence of the nonlinearity of derivative type $|\partial_t u|^p$, provided by the lower bound estimate \eqref{First_Lower_Bound_1}, with the iteration frame in \eqref{Iteration_Frame}, which has been derived thanks to the presence of the power nonlinearity $|u|^q$ on the right-hand side of our semilinear model.

On the other hand, as we pointed out in Remark \ref{Rem lb}, we know also other two lower bound estimates for $U$, given by \eqref{First_Lower_Bound_2} and \eqref{First_Lower_Bound_3}, respectively (let us recall that, in order to derive \eqref{First_Lower_Bound_3}, we need to require $u_1$ nonnegative and nontrivial). However, if we do employ one among these first lower bounds for $U$, somehow we neglect the effect $|\partial_t u|^p$ on the equation, obtaining the same blow-up result for the Tricomi equation with power nonlinearity $|u|^q$ from \cite{HWY17,LinTu19} in the subcritical case.

By using \eqref{First_Lower_Bound_2} in place of \eqref{First_Lower_Bound_1} and applying the machinery from Section \ref{Subsection iter arg}, we find the following result which coincides with Theorem 1.2 in \cite{LinTu19}.

\begin{corollary}  Let us consider $n\geqslant 2$ and $p,q>1$ satisfying $$1<q<p_0(n;\ell),$$ where $p_0(n;\ell)$ is the positive root of the quadratic equation in \eqref{exponent p0 He-Witt-Yin}.
 Let us assume that $u_0\in H^1(\mathbb{R}^n)$ and $u_1\in L^2(\mathbb{R}^n)$ are nonnegative, not both trivial and compactly supported functions with supports contained in $B_R$ for some $R>0$. Let $$u\in \mathcal{C}([0,T), H^1(\mathbb{R}^n)) \cap \mathcal{C}^1([0,T), L^2(\mathbb{R}^n))\cap L^q_{\mathrm{loc}}((0,T)\times \mathbb{R}^n) \ \ \mbox{such that} \ \ \partial_t u \in  L^p_{\mathrm{loc}}((0,T)\times \mathbb{R}^n) $$ be an energy solution to \eqref{Semi_Tri_Combined} according to Definition \ref{Defn_Energy_Solution} with lifespan $T=T(\varepsilon)$.

Then, there exists a positive constant $\varepsilon_0=\varepsilon_0(n,\ell,q,u_0,u_1,R)$ such that for any $\varepsilon\in (0,\varepsilon_0]$ the energy solution $u$ blows up in finite time. Furthermore, the upper bound estimate for the lifespan
\begin{align} \label{lifespan estimate cor 1}
T(\varepsilon)\leqslant C \varepsilon^{-\frac{q(q-1)}{\gamma(n,\ell,q)}}
\end{align}
holds, where the positive constant $C$ is independent of $	\varepsilon$ and
\begin{align} \label{def gamma}
\gamma(n,\ell,q)\doteq (\ell+1)+\tfrac 12 \big((\ell+1)n+1-2\ell\big)q - \tfrac 12 \big((\ell+1)n-1\big) q^2 .
\end{align}
\end{corollary}

Using \eqref{First_Lower_Bound_3} rather than \eqref{First_Lower_Bound_2}, we can prove a blow-up result for \eqref{Semi_Tri_Combined} for $1<q<p_{\mathrm{Kat}}((\ell+1)n)$, where $p_{\mathrm{Kat}}(n)=\frac{n+1}{n-1}$. 
This result provides actually an improvement of the upper bound for $q$ in the blow-up range only in the one dimensional case. More precisely, we get the following result, which is already know in the literature (see \cite[Remark 1.6]{HWY20}) for the power nonlinearity $|u|^q$.

\begin{corollary} \label{Cor n=1} Let us consider $n=1$ and $p,q>1$ satisfying $$1<q<1+\frac{2}{\ell}=p_{\mathrm{Kat}}(\ell+1).$$
 Let us assume that $u_0\in H^1(\mathbb{R})$ and $u_1\in L^2(\mathbb{R})$ are nonnegative and compactly supported functions with supports contained in $[-R,R]$ for some $R>0$. Additionally, we require that $$\int_{\mathbb{R}}u_1(x) \, \mathrm{d}x>0.$$ Let $u\in \mathcal{C}([0,T), H^1(\mathbb{R})) \cap \mathcal{C}^1([0,T), L^2(\mathbb{R}))\cap L^q_{\mathrm{loc}}((0,T)\times \mathbb{R})$ such that $ \partial_t u \in  L^p_{\mathrm{loc}}((0,T)\times \mathbb{R})$ be an energy solution to \eqref{Semi_Tri_Combined} according to Definition \ref{Defn_Energy_Solution} with lifespan $T=T(\varepsilon)$.

Then, there exists a positive constant $\varepsilon_0=\varepsilon_0(\ell,q,u_0,u_1,R)$ such that for any $\varepsilon\in (0,\varepsilon_0]$ the energy solution $u$ blows up in finite time. Furthermore, the upper bound estimate for the lifespan
\begin{align} \label{lifespan estimate cor n=1}
T(\varepsilon)\leqslant C \varepsilon^{-\left(\frac{q+1}{q-1}-(\ell+1)\right)^{-1}}
\end{align}
holds, where the positive constant $C$ is independent of $	\varepsilon$.
\end{corollary}

Besides, the upper bound for the lifespan in \eqref{lifespan estimate cor 1} can be improved for $n=2$ for certain exponents $q$, provided that the integral of $u_1$ over the whole space is a positive quantity.

\begin{corollary} \label{Cor n=2} Let $n=2$ and let us consider $p,q>1$ satisfying $$1<q<\frac{2(\ell+1)}{2\ell+1}.$$
 Let us assume that $u_0\in H^1(\mathbb{R}^2)$ and $u_1\in L^2(\mathbb{R}^2)$ are nonnegative and compactly supported functions with supports contained in $B_R$ for some $R>0$. Additionally, we require that $$\int_{\mathbb{R}^2}u_1(x) \, \mathrm{d}x>0.$$ Let $u\in \mathcal{C}([0,T), H^1(\mathbb{R}^2)) \cap \mathcal{C}^1([0,T), L^2(\mathbb{R}^2))\cap L^q_{\mathrm{loc}}((0,T)\times \mathbb{R}^2)$ such that $\partial_t u \in  L^p_{\mathrm{loc}}((0,T)\times \mathbb{R}^2) $ be an energy solution to \eqref{Semi_Tri_Combined} according to Definition \ref{Defn_Energy_Solution} with lifespan $T=T(\varepsilon)$.

Then, there exists a positive constant $\varepsilon_0=\varepsilon_0(\ell,q,u_0,u_1,R)$ such that for any $\varepsilon\in (0,\varepsilon_0]$ the energy solution $u$ blows up in finite time. Furthermore, the upper bound estimate for the lifespan
\begin{align} \label{lifespan estimate cor 2}
T(\varepsilon)\leqslant C \varepsilon^{-\left(\frac{q+1}{q-1}-2(\ell+1)\right)^{-1}}
\end{align}
holds, where the positive constant $C$ is independent of $	\varepsilon$.
\end{corollary}

The improvement of the upper bound for the lifespan in the  low dimensional case $n=1$ and $n=2$ from Corollaries \ref{Cor n=1} - \ref{Cor n=2} has already been observed for the case of the semilinear wave equation (see \cite{Tak15} and   \cite{LT18Scatt,PalTak19} for the semilinear damped wave equation in the scattering producing case).

\begin{remark} For $n\geqslant 3$ we cannot improve the upper bound estimate in \eqref{lifespan estimate cor 1}. As we have already pointed out,  employing \eqref{First_Lower_Bound_3} rather than \eqref{First_Lower_Bound_2}, we get a blow-up result for \eqref{Semi_Tri_Combined} provided that $1<q<p_{\mathrm{Kat}}((\ell+1)n)$. Moreover, the next upper bound estimate for the lifespan can be proved:
\begin{align*}
T(\varepsilon)\leqslant C \varepsilon^{-\left(\frac{q+1}{q-1}-(\ell+1)n\right)^{-1}}.
\end{align*} It turns out that this upper bound for $T(\varepsilon)$ improves the one in \eqref{lifespan estimate cor 1} for $$q<\frac{2(\ell+1)}{(\ell+1)n-1}.$$ However, $2(\ell+1)/\big((\ell+1)n-1\big)<1$ for $n\geqslant 3$ and any $\ell>0$, therefore, it makes sense to talk about an improvement of the lifespan estimate in \eqref{lifespan estimate cor 1} just in space dimensions $n=1,2$.
\end{remark}

\section{Analysis of the obtained results} \label{Section Final Rem}

Let us define the following region: 
\begin{align*}
\Gamma(n,\ell)\doteq\left\{(p,q)\in (1,\infty)^2:\, \Big[\big((\ell+1)n-1\big)p-2\ell(p-1)-2\Big](q-1)<4 \right\}.
\end{align*}
According to our main result, we proved the blow-up of local in time energy solutions to \eqref{Semi_Tri_Combined} under suitable sign assumptions for the Cauchy data provided that $(p,q)\in\Gamma(n,\ell)$.  We observe that, given $n\geqslant 2$, $\{\Gamma(n,\ell)\}_{\ell\geqslant 0}$ is a family of sets decreasing by inclusion, so the smaller $\ell$ the wider $\Gamma(n,\ell)$.

We begin by pointing out that for $\ell=0$, namely, for the case of the classical semilinear wave equation with combined nonlinearity, our result coincides exactly with the sharp result obtained by \cite{HZ14,HWY16}.

We focus now on the case $n\geqslant 2$ (we will investigate separately later the case $n=1$). In this case, we remark that the range $\Gamma(n,\ell)$ is not fully contained in the strips $$\left\{(p,q)\in (1,\infty)^2:\, p<p_1(n;\ell)  \quad \mbox{or} \quad  q<p_0(n; \ell) \right\},$$ delimited by $p=p_1(n;\ell)$, $q=p_0(n;\ell)$, which intersect in $\mathrm{S}\doteq (p_1(n;\ell),p_0(n;\ell))$, see Fig \ref{Fig 1}.

In other words, thanks to the combined presence of a power nonlinearity and of a nonlinearity of derivative type on the right-hand side of \eqref{Semi_Tri_Combined}, we find an enlargement of the blow-up range in comparison to the ranges for the corresponding semilinear models with either a nonlinearity of power type or a nonlinearity of derivative type. Let us consider for example the special case $p=q$. The intersection of $\partial\Gamma(n,\ell)$ with the diagonal yields the point $(p_{\mathrm{diag}}(n;\ell),p_{\mathrm{diag}}(n;\ell))$, where $p_{\mathrm{diag}}(n;\ell)$ is the positive solution to the quadratic equation 
\begin{align}\label{p diag eq}
\big((\ell+1)n-1-2\ell\big)p^2-\big((\ell+1)n+1-4\ell\big)p-2(\ell+1)=0.
\end{align}  By straightforward computations it follows that $p_1(n;\ell)<p_0(n;\ell)<p_{\mathrm{diag}}(n;\ell)$ for any $\ell>0$ and $n\geqslant 2$. Therefore, already on the diagonal we are able to prove a blow-up result for \eqref{Semi_Tri_Combined} for pairs $(p,q)$ satisfying $p>p_1(n;\ell)$ and $q>p_0(n;\ell)$. Since for $\ell=0$ it holds $p_0(n;0)=p_{\mathrm{diag}}(n;0)=p_{\mathrm{Str}}(n)$, we notice that this is a completely new phenomenon in comparison to the case of the classical semilinear wave equation with the same kind of nonlinear term. 

\begin{remark} Let us remark that $p_{\mathrm{diag}}(n;\ell)$ belongs somehow to the family of \emph{Strauss-type} exponents. The Strauss exponent $p_{\mathrm{Str}}(n)$ is the greatest root of the quadratic equation
\begin{align}\label{Strauss eq}
(n-1)(p-1)^2+(n-3)(p-1)-4=0,
\end{align} and it is the critical exponent for the semilinear wave equation with power nonlinearity. 
We are interested here in quadratic equations of the following kind:
\begin{align}\label{Eq alpha beta}
\alpha (p-1)^2+\beta (p-1)-4=0	
\end{align} or, equivalently, 
\begin{align*}
\alpha (p-1)=\alpha -\beta +\frac{4-(\alpha-\beta)}{p} .
\end{align*}
 The previous relation emphasizes the importance of the quantities $\alpha$ and $\alpha-\beta$.
Over the last years, several papers have been devoted to the treatment of semilinear wave models with critical exponents that are translation shifts of $p_{\mathrm{Str}}(n)$ in the dimensional parameter (see \cite{DLR15,DabbLuc15,LTW17,IS17,PalRei18,Pal18odd,Pal18even,TL1709,PT18,DA20} for models with time-dependent coefficients and \cite{GKW19,DKS21} for models with space-dependent coefficients). Typically, these models present lower order terms with critical decay rates, beyond the principal part of the operator, which is the wave operator $\partial_t^2 -\Delta$, and in all cases the relation $\alpha-\beta=2$ is satisfied as in \eqref{Strauss eq}. If we consider a generalized Tricomi operator $\partial_t^2 -t^{2\ell}\Delta$ instead of the wave operator, the quadratic equation that reveals $p_0(n;\ell)$ is 
\begin{align} \label{p_0(n,l) eq modified}
((\ell+1)n-1)(p-1)^2+((\ell+1)n-3+2\ell)(p-1)-4=0.
\end{align} In the last equation, we notice a rescaling in $\alpha$ due to the shape of the light-cone. Moreover, $\alpha-\beta=2(1-\ell)$. If we rewrite \eqref{p diag eq} in the following way:
\begin{align} \label{p diag eq modified}
((\ell+1)n-1-2\ell)(p-1)^2+((\ell+1)n-3)(p-1)-4=0,
\end{align} we observe that $\alpha-\beta=2(1-\ell)$ also in this case, so that, $p_{\mathrm{diag}}(n;\ell)$ is obtained by $p_0(n;\ell)$ through a shift of magnitude $-2\ell$ in the coefficients $\alpha$ and $\beta$. In this sense, $p_{\mathrm{diag}}(n;\ell)$ belongs to the family of Strauss-type exponents.


In Table \ref{Table}, we summarize the values of parameters $\alpha$, $\beta$, $\alpha-\beta$ for \eqref{Eq alpha beta} related to the previously discussed semilinear hyperbolic equations.  

On the other hand, taking the values of $\alpha$ in Table \ref{Table}, by the relation $\frac{\alpha}{2}(p-1)=1,$ we find the Glassey-type exponents $p_{\mathrm{Gla}}(n)$, $p_{\mathrm{Gla}}((\ell+1)n)$ and $p_{\mathrm{Gla}}((\ell+1)n-2\ell)$ for wave and generalized Tricomi operators (see Remark \ref{Rem on Lai-Schiavone result}).

\begin{table}[http]
	\renewcommand\arraystretch{1.5}
	\begin{center}  
		\begin{tabular}{|c|c|c|c|c|}  
			\hline  
			Semilinear PDE &  critical $p$ &  $\alpha$ & $\beta$ &  $\alpha-\beta$ \\ \hline  
			$\ \ \ \ \,\, (\partial_t^2-\Delta)u=|u|^p$ & $p_{\mathrm{Str}}(n)$ & $n-1$ & $n-3$ & $2$ \\ \hline  
			$(\partial_t^2-t^{2\ell}\Delta)u=|u|^p$ & $p_0(n;\ell)$ & $(\ell+1)n-1$ & $(\ell+1)n-3+2\ell$ & $2(1-\ell)$ \\  \hline
			$\qquad\quad \ \, (\partial_t^2-t^{2\ell}\Delta)u=|\partial_tu|^p+|u|^p$ & $p_{\mathrm{diag}}(n;\ell)$ & $(\ell+1)n-1-2\ell$ & $(\ell+1)n-3$ & $2(1-\ell)$  \\  \hline
		\end{tabular}  
	\end{center}  
	\caption{Values of $p$, $\alpha$, $\beta$ and $\alpha-\beta$ in the quadratic equation \eqref{Eq alpha beta} for different semilinear PDEs}  \label{Table}
\end{table}  
\end{remark}

Let us determine now the intersection of the boundary of the blow-up region, i.e. $\partial \Gamma(n,\ell)$, with the straight lines with equations $q=p_0(n;\ell)$ and $p=p_1(n;\ell)$. We get started with the intersection of $\partial \Gamma(n,\ell)$ with the horizontal line. Denoting $\mathrm{Q}=(\widetilde{p}_0(n;\ell),p_0(n;\ell))$ the point such that $$\partial \Gamma(n,\ell)\cap \{q=p_0(n;\ell)\}=\{\mathrm{Q}\}$$ in the $(p,q)$ - plane, if we combine the relation $\big[\big((\ell+1)n-1\big)p-2\ell(p-1)-2\big](q-1)=4$ and \eqref{p_0(n,l) eq modified}, then, we obtain that $\widetilde{p}_0(n;\ell)$ is the greatest root of the quadratic equation
\begin{align*}
((\ell+1)n-1-2\ell) (p-1)^2+ ((\ell+1)n-3-2\ell) (p-1)-2(\ell+2)-\frac{4\ell(\ell+1)}{(\ell+1)n-1-2\ell}=0.
\end{align*}
We underline explicitly that 

$$\widetilde{p}_0(n;0)=  p_{\mathrm{Str}}(n)=p_{\mathrm{diag}}(n;0)=p_0(n;0).$$
So, denoting $\mathrm{D}\doteq (p_{\mathrm{diag}}(n;\ell),p_{\mathrm{diag}}(n;\ell))$, we find that $\mathrm{Q},\mathrm{D}\to (p_{\mathrm{Str}}(n),p_{\mathrm{Str}}(n))$ as $\ell\to 0$. The fact that $\mathrm{Q},\mathrm{D}$ collapse on the same point as $\ell\to 0$ allows us to understand better the above pointed out phenomenon of the enlargement for $\ell>0$ of the blow-up range on the diagonal. 
Indeed, for the semilinear Cauchy problem 
\begin{align*}
\begin{cases} 
\partial_t^2 u- t^{2\ell}\Delta u =|\partial_t u|^p+|u|^p, &  x\in \mathbb{R}^n, \ t>0,\\
u(0,x)=\varepsilon u_0(x), & x\in \mathbb{R}^n, \\ 
\partial_t u(0,x)=\varepsilon u_1(x), & x\in \mathbb{R}^n,
\end{cases}
\end{align*} 
the blow-up range 
is extended up to the wider range $1<p<p_{\mathrm{diag}}(n;\ell)$, in comparison to the corresponding semilinear Cauchy problems with either power nonlinearity $|u|^p$ or nonlinearity of derivative type $|\partial_t u|^p$. We recall that for $\ell=0$ this circumstance does not occur.


We may rewrite
\begin{align} \label{Boundary Gamma region}
\partial\Gamma(n,\ell)= \left\{p\in (1,\infty): q=f(p;n,\ell)  \right\} \quad \mbox{with} \quad f(p;n,\ell)\doteq 1+\frac{4}{((\ell+1)n-1-2\ell)p+2\ell-2}.
\end{align} Using that the function $f= f(p;n,\ell)$ in \eqref{Boundary Gamma region} is strictly decreasing with respect to $p$, the relation $p_0(n;\ell)<p_{\mathrm{diag}}(n;\ell)$ and that the point $\mathrm{D}$ is given by the intersection of $\partial\Gamma(n,\ell)$ with the diagonal $\{p=q\}$, we get $p_{\mathrm{diag}}(n;\ell)<\widetilde{p}_0(n;\ell) $.

Next we study the intersection of $\partial \Gamma(n,\ell)$ with the vertical line $p=p_1(n;\ell)$. 
 Neglecting the influence of the power nonlinearity $|u|^q$ and following \cite{LP20}, a blow-up result for $1<p\leqslant p_1(n;\ell)$ can be shown in the case with combined nonlinearity too.

For $p=p_1(n;\ell)$ on $\partial \Gamma(n,\ell)$ we find $$q=\widetilde{p}_1(n;\ell) \doteq \frac{((\ell+1)n+3)((\ell+1)n-1)-4\ell}{((\ell+1)n-1)^2-4\ell}. 
$$  Since the function $f= f(p;n,\ell)$ in \eqref{Boundary Gamma region} is strictly decreasing with respect to $p$, $p_1(n;\ell)<p_{\mathrm{diag}}(n;\ell)$ and the point $\mathrm{D}$ is given by the intersection of $\partial\Gamma(n,\ell)$ with the diagonal $\{p=q\}$, it follows $p_{\mathrm{diag}}(n;\ell)<\widetilde{p}_1(n;\ell) $.

%
%
%

Hence, if we denote by $\mathrm{P}$ the point in the $(p,q)$ - plane such that 
\begin{align*}
& \partial \Gamma(n,\ell) \cap \{p=p_1(n;\ell)\}=\{\mathrm{P}\}, 
\end{align*}	 then, from left to right we have the sorting $\mathrm{P},\mathrm{D},\mathrm{Q}$ on the branch of the hyperbola given by $\partial \Gamma(n,\ell)$. In other words, we obtained
\begin{align*}
& p_1(n;\ell)<p_{\mathrm{diag}}(n;\ell)<\widetilde{p}_0(n;\ell), \\ & p_0(n;\ell)<p_{\mathrm{diag}}(n;\ell)<\widetilde{p}_1(n;\ell).
\end{align*}

Summarizing, 
we extend the blow-up region for \eqref{Semi_Tri_Combined}  in the $(p,q)$ - plane to the curvilinear triangle $\mathrm{PQS}$. 
In Fig. \ref{Fig 1} we collect all properties that we discussed on the blow-up region for \eqref{Semi_Tri_Combined} in this section.



\begin{remark}\label{Rem on Lai-Schiavone result}
Let us provide a wider overview on blow-up results for the semilinear generalized Tricomi equation with nonlinearity of derivative type. In \cite{LS20} the authors found 
 $p_{\mathrm{Gla}}((\ell+1)n-2\ell)$ as upper bound for $p$ in the blow-up range for $n\geqslant 2$. 
For combined nonlinearity, neglecting $|u|^q$ one deduces a blow-up result assuming $1<p\leqslant p_{\mathrm{Gla}}((\ell+1)n-2\ell)$. Plugging $p=p_{\mathrm{Gla}}((\ell+1)n-2\ell)$ into the equation for $\partial \Gamma(n,\ell)$, we get  $q=p_{\mathrm{conf}}((\ell+1)n)$, where $p_{\mathrm{conf}}(n)\doteq \frac{n+3}{n-1}$ is the \emph{conformal exponent} for the classical semilinear wave equation (cf. \cite{LinSog95}). 
\end{remark}

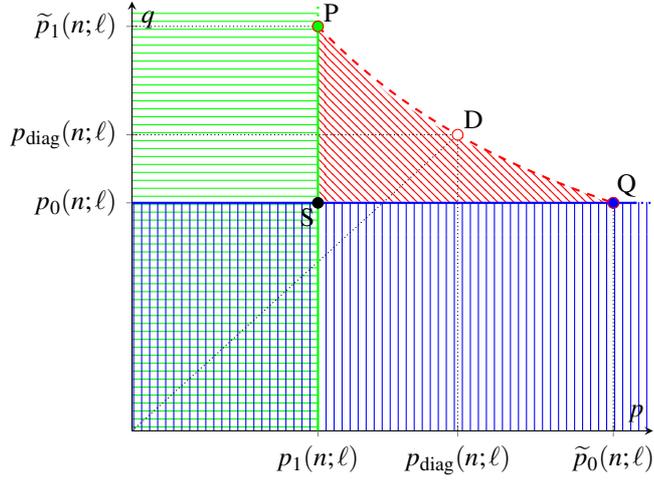
\begin{figure}[h]
\centering
\begin{tikzpicture}
\begin{axis} [xmin=1,xmax=2.6,ymin=1,ymax=2.45, axis x line=middle,axis y line=middle, 
 xtick={1.571,2,2.478},ytick={1.770,2,2.366},
 xticklabels={$p_1(n;\ell)$,$p_{\mathrm{diag}}(n;\ell)$, $\widetilde{p}_0(n;\ell)$},
yticklabels={$p_0(n;\ell)$,$p_{\mathrm{diag}}(n;\ell)$,$\widetilde{p}_1(n;\ell)$},
 xlabel=$p$,ylabel=$q$, legend style={draw=none}] 
 \fill [pattern=
 horizontal lines, pattern color=green]  (0.3,0.3) -- (57,0.3) -- (57,143) --(0.3,143);
 \fill [pattern=
vertical lines, pattern color=blue]  (0.3,0.3) -- (159,0.3) -- (159,77) --(0.3,77);

\fill [pattern=north west lines, pattern color=red]  (57.5,77.5) -- (146,77) -- (99,99) -- (79.5,113) -- (57,135.5); 

 \addplot [domain=1.571:2.478,samples=40,
 smooth,thick,red,dashed]{1+(4)/(2.5*x-1)};
 \addplot [domain=1:2,samples=40,
 smooth,thin,black,densely dotted]{x};  
 \addplot [domain=1:2.55,samples=40,smooth,thick,blue]{1.770}; 
 \addplot [domain=2.54:2.6,
 samples=40,smooth,thick,blue,dotted]{1.770}; 
 \addplot
[domain=1:2.35,variable=\t,
 samples=40,smooth,thick,green]
({1.571},{t}); 
\addplot
[domain=2.35:2.45,variable=\t,
 samples=40,smooth,thick,green,dotted]
({1.571},{t}); 

 \addplot
[domain=1:2,variable=\t,
 samples=40,smooth,thin,densely dotted]
({2},{t}); 
 \addplot
[domain=1:1.770,variable=\t,
 samples=40,smooth,thin,densely dotted]
({2.478},{t}); 
\addplot [domain=1:2,samples=40,smooth,thin,densely dotted]{2}; 
\addplot [domain=1:1.571,samples=40,smooth,thin,densely dotted]{2.366};

\addplot [only marks,mark=*] coordinates { (1.571,1.770) };
\coordinate[label=below left: $\mathrm{S}$]
					(P)  at (59,78.5);
\addplot [mark=*,green] coordinates { (1.571,2.366) };
\addplot [mark=o,red] coordinates { (1.571,2.366) };
\coordinate[label= above right: $\mathrm{P}$]
					(P) at (56,135);	
\addplot [mark=*,blue] coordinates { (2.478,1.770) };
\addplot [mark=o,red] coordinates { (2.478,1.770) };
\coordinate[label= above right: $\mathrm{Q}$]
					(P) at (146,76);	
					
					
\addplot [mark=*,white] coordinates { (2,2) };
\addplot [mark=o,red] coordinates { (2,2)};
\coordinate[label= above right: $\mathrm{D}$]
					(P) at (99,99);			
							
\end{axis};
\end{tikzpicture}
\caption{Blow-up range for \eqref{Semi_Tri_Combined}: case $n\geqslant 2$}
\label{Fig 1}
\end{figure}

Moreover, collecting the upper bound estimates from Theorem \ref{Theorem}, Section \ref{Section |u|^q} and \cite{LinTu19,LP20} we have 
\begin{align*}
T(\varepsilon) \leqslant 
\begin{cases}  C \varepsilon^{-\frac{p(q-1)}{\theta(n,\ell,p,q)}}  & \mbox{if} \ (p,q)\in\Gamma(n,\ell); \\ 
C \varepsilon^{-\frac{q(q-1)}{\gamma(n,\ell,q)}}  & \mbox{if} \ 1<q<p_0(n;\ell) \ \mbox{and}\ n\geqslant 2; \\
 C \varepsilon^{-\left(\frac{q+1}{q-1}-(\ell+1)n\right)^{-1}} & \mbox{if} \ 1<q<p_{\mathrm{Kat}}((\ell+1)n) \ \mbox{and} \ n=1,2;  \\
 C \varepsilon^{-\left(\frac{1}{p-1}-\frac{(\ell+1)n-1}{2}\right)^{-1}} & \mbox{if} \ 1<p<p_{\mathrm{Gla}}((\ell+1)n); \\
\exp\big(C \varepsilon^{-q(q-1)}\big) & \mbox{if} \ q=p_0(n;\ell); \\
\exp\big(C \varepsilon^{-(p-1)} \big)& \mbox{if} \ p=p_{\mathrm{Gla}}((\ell+1)n).
\end{cases}
\end{align*}


 For the the lifespan estimate in the critical case $q=p_0(n;\ell)$ we refer to \cite[Theorem 1.3]{LinTu19}, while for the upper bound estimates in the case $1<p\leqslant p_{\mathrm{Gla}}((\ell+1)n)$ we address the interested reader to \cite[Theorem 1.1]{LP20}.

%

Finally, let us consider more in detail the case $n=1$.
In the  one dimensional case the blow-up condition \eqref{blow up range p,q} is 
\begin{align}\label{blow up range p,q n=1}
(-\ell p +2\ell-2)(q-1)<4.
\end{align} The main difference in comparison to the higher dimensional cases is that the constant that multiplies $p$ in \eqref{blow up range p,q n=1} is negative, modifying substantially the shape of the blow-up region. We will consider separately the cases $\ell\in (0,2]$, $\ell\in (2,4]$ and $\ell\in (4,\infty)$. 

 For the semilinear wave equations ($\ell=0$) with power nonlinearity $|u|^q$ and with nonlinearity of derivative type $|\partial_t u|^p$, respectively,  it is well-know in the literature (cf. \cite{Kato80,Zhou01}) that for $n=1$ blow-up results for local in time solutions hold for any $q>1$ and any $p>1$, respectively. Therefore, for the Cauchy problem associated with semilinear wave equation with combined nonlinearity in the 1-d case $\partial_t^2u-\partial_x^2 u=|\partial_t u|^p+|u|^q$ for any $p,q>1$ it is possible to prove the blow-up in finite time of the corresponding local in time solution, provided that suitable sign assumptions are required for the Cauchy data.  In the first two subcases, namely for $\ell\in (0,4]$, also for \eqref{Semi_Tri_Combined} the same situation occurs as for the corresponding wave equation ($\ell=0$), that is, the blow-up range coincides with the entire set $\{p,q>1 \}$. However, in order to prove this fact we need to distinguish  the case $\ell\in (0,2]$ from the case $\ell\in (2,4]$.
 
  For $\ell\in (0,2]$ it holds
 \begin{align*}
 \{(p,q)\in \mathbb{R}^2: p,q>1\} \subset \{(p,q)\in\mathbb{R}^2: (-\ell p +2\ell-2)(q-1)<4\},
 \end{align*}
  so it follows immediately that the blow-up range from Theorem \ref{Theorem} is the full set $\{p,q>1 \}$, covering the case $p>p_{\mathrm{Gla}}(\ell+1)$ and $q>p_{\mathrm{Kat}}(\ell+1)$ as well (see Fig. \ref{Fig 2}). We point out that, as in the one dimensional case the critical exponent for the semilinear generalized Tricomi equation with power nonlinearity is $p_{\mathrm{Kat}}(\ell+1)$, we  consider the point $\mathrm{S}'\doteq (p_{\mathrm{Gla}}(\ell+1),p_{\mathrm{Kat}}(\ell+1))$ in place of $\mathrm{S}$ in the figures of the blow-up regions for $n=1$.
 
\color{black} 
 


\begin{figure}[h]
\centering
\begin{tikzpicture}
\begin{axis} [xmin=1,xmax=5,ymin=1,ymax=5, axis x line=middle,axis y line=middle, 
 xtick={3},ytick={3},
 xticklabels={$p_{\mathrm{Gla}}(\ell+1)$},
yticklabels={$p_{\mathrm{Kat}}(\ell+1)$},
 xlabel=$p$,ylabel=$q$, legend style={draw=none}] 
 \fill [pattern=horizontal lines, pattern color=green]  (0.3,0.3) -- (200,0.3) -- (200,400) --(0.3,400);
 \fill [pattern=vertical lines, pattern color=blue]  (0.3,0.3) -- (400,0.3) -- (400,200) --(0.3,200);
\fill [pattern=north west lines, pattern color=red] (200,200) -- (400,200) -- (400,400) -- (200,400);
 \addplot [domain=1:4.85,samples=40,smooth,thick,blue]{3}; 
 \addplot [domain=4.85:5,
 samples=40,smooth,thick,blue,dotted]{3}; 
 \addplot
[domain=1:4.85,variable=\t,
 samples=40,smooth,thick,green]
({3},{t}); 
\addplot
[domain=4.85:5,variable=\t,
 samples=40,smooth,thick,green,dotted]
({3},{t}); 
\addplot [only marks,mark=*] coordinates { (3,3) };
\coordinate[label=above left: $\mathrm{S}'$]
					(P)  at (200,200);

\end{axis};
\end{tikzpicture}
\caption{Blow-up range for \eqref{Semi_Tri_Combined}: case $n=1$, $\ell\in (0,2]$}
\label{Fig 2}
\end{figure}
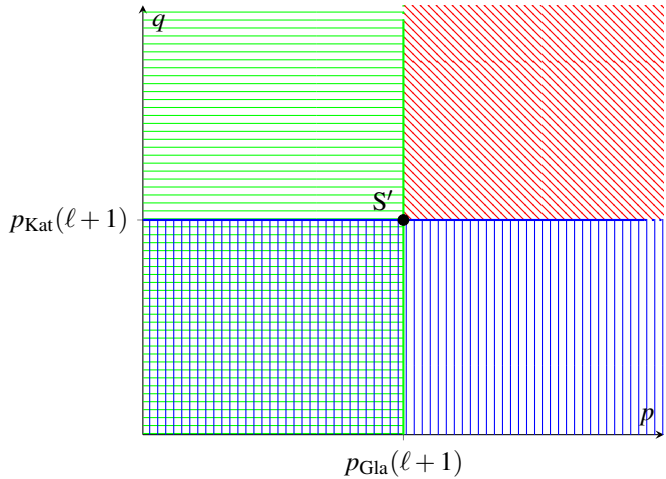


On the other hand, for $\ell\in (2,4]$ the hyperbola
\begin{align*}
\{(p,q)\in \mathbb{R}^2: (-\ell p +2\ell-2)(q-1)=4\}
\end{align*}
 has asymptotes $q=1$ and $p=2-\frac{2}{\ell}>1$. Hence, the condition in \eqref{blow up range p,q n=1} does not cover the entire set $\{p,q>1 \}$. Nevertheless, $2-\frac{2}{\ell}<1+\frac{2}{\ell}=p_{\mathrm{Gla}}(\ell+1)$, so, using the blow-up result for $1<p\leqslant p_{\mathrm{Gla}}(\ell+1)$ that can be proved by working just with the nonlinearity of derivative type (see \cite[Theorem 1.1]{LP20}), we close the gap left by Theorem \ref{Theorem} (cf. Fig. \ref{Fig 3}).



\begin{figure}[h]
\centering
\begin{tikzpicture}
\begin{axis} [xmin=1,xmax=2,ymin=1,ymax=3, axis x line=middle,axis y line=middle, 
 xtick={1.429,1.571},ytick={1.571,2},
 xticklabels={$2-\frac{2}{\ell}\qquad$,$p_{\mathrm{Gla}}(\ell+1)$},
yticklabels={$p_{\mathrm{Kat}}(\ell+1)$,$\frac{\ell+2}{\ell-2}$},
 xlabel=$p$,ylabel=$q$, legend style={draw=none}] 
 \fill [pattern=horizontal lines, pattern color=green]  (0.3,0.3) -- (57.1,0.3) -- (57.1,20) --(0.3,20);
 \fill [pattern=vertical lines, pattern color=blue]  (0.1,0.1) -- (100,0.1) -- (100,5.7) --(0.1,5.7);
\fill [pattern=north west lines, pattern color=red]  (57.1,5.7) -- (100,5.7) -- (100,20) -- (57.1,20); 

 \addplot [domain=1:1.410,samples=40,
 smooth,thick,red,dashed]{1+(1.5)/((-3.5)*x+5)};
 \addplot [domain=1:1.95,samples=40,smooth,thick,blue]{1.571}; 
 \addplot [domain=1.95:2,
 samples=40,smooth,thick,blue,dotted]{1.571}; 
 \addplot
[domain=1:2.9,variable=\t,
 samples=40,smooth,thick,green]
({1.571},{t}); 
\addplot
[domain=2.9:3,variable=\t,
 samples=40,smooth,thick,green,dotted]
({1.571},{t}); 

\addplot
[domain=1:3.9,variable=\t,
 samples=40,smooth,thin,densely dotted]
({1.429},{t}); 
\addplot
[domain=3.9:4,variable=\t,
 samples=40,smooth,thin,densely dotted]
({1.429},{t}); 

\addplot [only marks,mark=*] coordinates { (1.571,1.571) };
\coordinate[label=above left: $\mathrm{S}'$]
					(P)  at (57.1,5.7);
	
\end{axis};
\end{tikzpicture}
\caption{Blow-up range for \eqref{Semi_Tri_Combined}: case $n=1$, $\ell\in (2,4]$}
\label{Fig 3}
\end{figure}
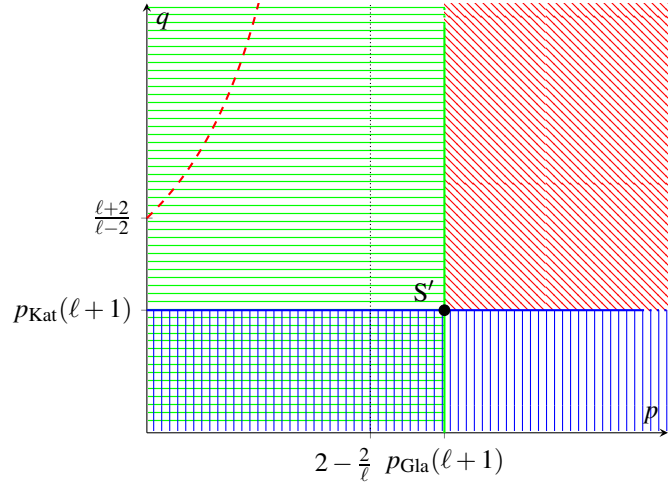


Finally, for $\ell\in (4,\infty)$, even combining the results of Theorem \ref{Theorem} and \cite[Theorem 1.1]{LP20}, we do not obtain $\{p,q>1 \}$ as blow-up range for \eqref{Semi_Tri_Combined}. More precisely, we have to exclude the region
\begin{align*}
\big\{(p,q)\in (1,\infty)^2: p>p_{\mathrm{Gla}}(\ell+1) \ \ \mbox{and} \ \ (-\ell p +2\ell-2)(q-1)\geqslant 4 \big\},
\end{align*} contained in the strip $\{p_{\mathrm{Gla}}(\ell+1)<p<2-\frac{2}{\ell}\}$, from the blow-up range, since in this case the asymptote $p=2-\frac{2}{\ell}$ lies to the right of the critical value $p=p_{\mathrm{Gla}}(\ell+1)$ (see Fig. \ref{Fig 4}).



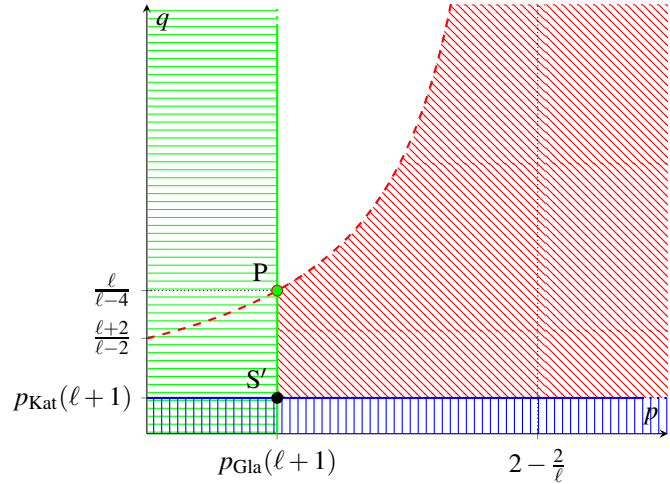
\begin{figure}[h]
\centering
\begin{tikzpicture}
\begin{axis} [xmin=1,xmax=2,ymin=1,ymax=4, axis x line=middle,axis y line=middle, 
 xtick={1.25,1.75},ytick={1.25,1.667,2},
 xticklabels={$p_{\mathrm{Gla}}(\ell+1)$,$2-\frac{2}{\ell}$},
yticklabels={$p_{\mathrm{Kat}}(\ell+1)$,$\frac{\ell+2}{\ell-2}$,$\frac{\ell}{\ell-4}$},
 xlabel=$p$,ylabel=$q$, legend style={draw=none}] 
 \fill [pattern=horizontal lines, pattern color=green]  (0.3,0.3) -- (25,0.3) -- (25,300) --(0.3,300);
 \fill [pattern=vertical lines, pattern color=blue]  (0.3,0.3) -- (100,0.3) -- (100,25) --(0.3,25);
\fill [pattern=north west lines, pattern color=red] (25,25) -- (100,25) -- (100,300) -- (58.3,300) -- (55,250) -- (50.5,200) -- (42,150) -- (35.5,125) -- (25,100);
\addplot [domain=1:1.70,samples=40,
 smooth,thick,red,dashed]{1+(4)/((-8)*x+14)};
\addplot [domain=1:1.95,samples=40,smooth,thick,blue]{1.25}; 
\addplot [domain=1.95:2,
 samples=40,smooth,thick,blue,dotted]{1.25};
\addplot [domain=1:1.25,
 samples=40,smooth,thin,densely dotted]{2}; 
\addplot
[domain=1:3.85,variable=\t,
 samples=40,smooth,thick,green]
({1.25},{t}); 
\addplot
[domain=3.85:4,variable=\t,
 samples=40,smooth,thick,green,dotted]
({1.25},{t}); 
\addplot
[domain=1:4,variable=\t,
 samples=40,smooth,thin,densely dotted]
({1.75},{t}); 
\addplot [only marks,mark=*] coordinates { (1.25,1.25) };
\coordinate[label=above left: $\mathrm{S}'$]
					(P)  at (25,25);
\addplot [mark=*,green] coordinates { (1.25,2) };
\addplot [mark=o,red] coordinates { (1.25,2) };
\coordinate[label= above left: $\mathrm{P}$]
					(P) at (25,100);	
\end{axis};
\end{tikzpicture}
\caption{Blow-up range for \eqref{Semi_Tri_Combined}: case $n=1$, $\ell\in (4,\infty)$}
\label{Fig 4}
\end{figure}

\paragraph*{Final note}
After the preparation of the final version of the present manuscript, the authors found out the existence of the preprint \cite{HH20T}, where the same blow-up result for \eqref{Semi_Tri_Combined} has been obtained independently.

\section*{Acknowledgments}

A. Palmieri is supported by the GNAMPA project `Problemi stazionari e di evoluzione nelle equazioni di campo nonlineari dispersive'. S. Lucente is supported by the PRIN 2017 project `Qualitative and quantitative aspects of nonlinear PDEs' and by the GNAMPA project `Equazioni di tipo dispersivo: teoria e metodi'.


\appendix

\section{Properties of the modified Bessel function of the second kind} \label{App A}

In this appendix, we recall the main properties of modified Bessel functions of the second kind that we employ throughout the paper. For further details we refer to \cite[Chapter 10]{NIST10}.

The modified Bessel function of the second kind of order $\nu$ satisfies the second-order linear differential equation with polynomial coefficients
\begin{align} \label{ODE Mod Bessel funct 2nd kind}
\tau^2 \mathrm{K}_{\nu}''(\tau)+\tau \mathrm{K}_{\nu}'(\tau)-(\nu^2+\tau^2)\mathrm{K}_{\nu}(\tau)=0
\end{align}

The derivative of the modified Bessel function of the second kind can be expressed through the following recursive relations:
\begin{equation} \label{recursive identity K' alpha}
\begin{split}
\mathrm{K}_\nu'(\tau) & = -\mathrm{K}_{\nu-1}(\tau)-\frac{\nu}{\tau} \, \mathrm{K}_\nu(\tau), \\
\mathrm{K}_\nu'(\tau) &  = -\mathrm{K}_{\nu+1}(\tau)+\frac{\nu}{\tau} \, \mathrm{K}_\nu(\tau) ,\\
 \mathrm{K}_\nu'(\tau) &  =  -\frac{1}{2}\big(\mathrm{K}_{\nu-1}(\tau)+\mathrm{K}_{\nu+1}(\tau)\big).
 \end{split}
\end{equation}
Finally, the following asymptotic estimates for $\mathrm{K}_\nu$ hold for small and large argument, respectively,
\begin{align} \label{asymptotic K alpha small}
\mathrm{K}_\nu(\tau) &\sim 2^{\nu-1} \Gamma(\nu) \ \tau^{-\nu} \ \ \quad \qquad \qquad \mbox{as} \ \ \tau\to 0^+,  \ \ \mbox{for} \ \ \Re\nu >0, \\
\label{asymptotic K alpha}
\mathrm{K}_\nu(\tau) & = \sqrt{\frac{\pi}{2\tau}} \, \mathrm{e}^{-\tau} \big(1+O(\tau^{-1})\big) \qquad \mbox{as} \ \ \tau\to \infty.
\end{align}

\end{document}